\documentclass[11pt]{article}
\font\smallit=cmti10

\usepackage{amssymb,latexsym,amsmath,epsfig,amsthm} 
\usepackage{hyperref}

\makeatletter
\newcommand{\tpmod}[1]{{\@displayfalse\pmod{#1}}}
\makeatother

\makeatletter

\renewcommand\section{\@startsection {section}{1}{\z@}
{-30pt \@plus -1ex \@minus -.2ex}
{2.3ex \@plus.2ex}
{\normalfont\normalsize\bfseries}}

\renewcommand\subsection{\@startsection{subsection}{2}{\z@}
{-3.25ex\@plus -1ex \@minus -.2ex}
{1.5ex \@plus .2ex}
{\normalfont\normalsize\bfseries}}

\renewcommand{\@seccntformat}[1]{\csname the#1\endcsname. }

\makeatother

\newtheorem{theorem}{Theorem}
\newtheorem{lemma}{Lemma}
\newtheorem{claim}{Claim}
\newtheorem{conjecture}{Conjecture}

\usepackage{amsfonts}
\usepackage{booktabs}
\usepackage{siunitx}

\begin{document}
\begin{center}
\uppercase{\bf Explicit algebraic numbers all whose integer parts of powers are always composite}

\vskip 20pt
{\bf Suyeong Hahn}\\
{\smallit Choate Rosemary Hall, Wallingford, CT, USA}\\
{\tt shahn27@choate.edu}\\
\vskip 10pt
{\bf Dan Ismailescu}\\
{\smallit Department of Mathematics, Hofstra University,
Hempstead, NY, USA.}\\
{\tt dan.p.ismailescu@hofstra.edu}\\
\vskip 10pt
{\bf Ganghun Kim }\\
{\smallit The Lawrenceville School, Lawrenceville, NJ, USA}\\
{\tt jasonghkim@gmail.com}\\
\vskip 10pt
{\bf Minseung Kim}\\
{\smallit Collegiate School, New York, NY, USA}\\
{\tt minibambi09@gmail.com}\\
\vskip 10pt

\end{center}



\begin{abstract}
We present several explicit classes of algebraic integers $\alpha>1$ such that $\lfloor \alpha^n\rfloor$ is composite for all but finitely many $n$.
\end{abstract}

\pagestyle{myheadings}
\thispagestyle{empty}
\baselineskip=12.875pt
\vskip 30pt


\section{Introduction}

One of the oldest and most challenging questions in number theory is to determine whether a given integer sequence contains infinitely many prime numbers.
The problem is still unsolved for seemingly very simple looking sequences such as $\{n^2+1\}$ or $\{F_n\}$, the Fibonacci sequence.

In this paper we study sequences which are integer parts of powers of a fixed real number, $\lfloor \alpha^n\rfloor, n\ge 1$.
Baker and Harman \cite{bh} proved that the sequence $\{\lfloor \alpha^n\rfloor\}$ contains infinitely many primes for almost all $\alpha>1$. On the other hand, they proved that there are uncountably many numbers $\alpha$ such that $\lfloor \alpha^n\rfloor$ is composite for all positive integers $n$. Their method does not allow to give explicit expressions for such $\alpha$.


Forman and Shapiro \cite{fs} proved that if $\alpha=3/2$ or $\alpha=4/3$ the sequence $\{\lfloor \alpha^n\rfloor\}$ contains infinitely many composite numbers. Dubickas and Novikas \cite{dn} showed that the same conclusion holds if $\alpha=5/4$. They also gave an explicit example of a transcendental number
$\alpha$ (obtained as the limit of a certain recurrent integer sequence) for which the sequence $\{\lfloor \alpha^n\rfloor\}$ has infinitely many terms in an arbitrary arithmetic progression.
Cass \cite{cass} showed that if $\alpha$ is a quadratic unit then $\{\lfloor \alpha^n\rfloor\}$ contains infinitely many composite numbers. This result was later extended by Dubickas \cite{dubickas1} to all Pisot and Salem numbers.

An argument of Huxley \cite{bh} shows that for every $n\in \mathbb{N}$, $\lfloor ((1+\sqrt[3]{2}+\sqrt[3]{4})^{2n}\rfloor$ is equal to $2$ or $3$ modulo $6$ and therefore composite for all $n\ge 1$. Alkauskas and Dubickas \cite{ad} extended Huxley's idea and constructed Pisot numbers of arbitrary degree such that the integer parts of their powers are always even, and therefore composite. For example, their method implies that $\lfloor ((5+\sqrt{17})/2)^n \rfloor$ is even for all $n\ge 1$.

In this study, we build upon the results presented in \cite{AaronMe} and provide several new classes of Pisot numbers $\alpha$ which have the property that $\lfloor \alpha^n\rfloor$ is composite for all but possibly finitely many $n\ge 1$. 

The paper is organized as follows.

Section \ref{sectionpisot} contains several basic facts regarding Pisot polynomials and Pisot numbers.

In Section \ref{sectionquadratic} we show the connection between quadratic Pisot numbers and Lucas sequences of the second kind. We then use the divisibility properties of the latter to construct several families of quadratic Pisot numbers which possess the desired property.

Section \ref{sectioncubic} contains examples of cubic Pisot numbers; in particular, we identify an infinite set of cube-free integers $m$ for which
$\lfloor \mu^n\rfloor$ is composite for all $n\ge 4$ where $\mu$ is the fundamental unit in the pure cubic field $\mathbb{Q}(\sqrt[3]{m})$.
In Section \ref{sectionbiquadratic} we exhibit several biquadratic Pisot numbers $\alpha= (c_1+c_2\sqrt{p}+c_3\sqrt{q}+c_4\sqrt{p}\sqrt{q})/2$ with the property that $\lfloor \alpha^n\rfloor$ is composite for all $n\ge 3$.

In Section \ref{sectiondegreed} we generalize a theorem of Alkauskas and Dubickas and prove that for any given integers $d\ge 2$, $p\ge 2$ there exists a Pisot polynomial of degree $d$, whose largest root $\alpha$ has the property that $\lfloor \alpha^n\rfloor$ is an integer multiple of $p$ for all $n\ge 1$. The paper concludes with a section containing several remarks and a few open questions.

\section{\bf Pisot polynomials and Pisot numbers}\label{sectionpisot}

A monic polynomial with integer coefficients, $P$, satisfying $P(0)\neq 0$, is said to be a \emph{Pisot
polynomial} if $P$ has a unique root, say $\alpha$, which is greater than $1$ and
all the other roots have modulus less than $1$. In this case, $P$ is irreducible over $\mathbb{Q}$, and
$\alpha$ is said to be a \emph{Pisot number}.

Pisot \cite{pisot} introduced these numbers in his 1938 thesis. Salem \cite{salem} proved that the set of Pisot numbers
is closed and conjectured that the smallest Pisot number is $\theta_1=1.324\ldots$, the positive root of $z^3-z-1=0$.
This was confirmed by Siegel \cite{siegel} who also showed that the second smallest Pisot number is $\theta_2=1.380\ldots$,
the positive root of $z^4-z^3-1=0$. Siegel also proved that all other Pisot numbers lie in the interval $(\sqrt{2}, \infty)$.

Clearly, if $\deg(P)=1$ then $P$ is a Pisot polynomial if and only if $P(0)\le -2$; similarly,
if $\deg(P)=2$ then $P$ is a Pisot polynomial if and only if $P(-1)>0$ and $P(1)<0$.
While in general there is no simple characterization of Pisot numbers, there exist fast algorithms which allow to test whether a given polynomial is Pisot
\cite{duffin, zaimi}.

\section{Quadratic Pisot numbers}\label{sectionquadratic}

Let $P(z)=z^2-az+b$ be a quadratic Pisot polynomial whose roots are $\alpha>1$ and $\beta \in (-1,1)$. As mentioned earlier, this is equivalent to
requiring $P(-1)=1+a+b>0$ and $P(1)=1-a+b<0$. In particular, $a^2-4b>0$ and $a\ge 1$.

Define $U_n=\alpha^n+\beta^n$, for all $n\ge 0$. It is immediate to note that $\{U_n\}$ is an integer sequence given by the initial conditions
$U_0=2$, $U_1=a$, and by the recurrence relation $U_n=aU_{n-1}-bU_{n-2}$ for all $n\ge 2$. This sequence is known as the Lucas sequence of the second kind with parameters $a$ and $b$. It can be easily shown that since $a\ge 1$ and $a^2-4b>0$ the sequence $\{U_n\}$ is strictly increasing for all $n\ge 1$.

Since $\beta\in(-1,1)$ it follows that $\lfloor \alpha^n\rfloor =U_n-1$ or  $\lfloor \alpha^n\rfloor =U_n$ depending on whether $\beta^n>0$ or $\beta^n<0$.
It is therefore useful to investigate the divisibility properties of the sequence $\{U_n\}$.

\begin{lemma}\label{lemmaLucas}
Let $a$ and $b$ be integers so that $a\ge 1, b\neq 0$ and $\Delta=a^2-4b>0$. Then, the Lucas sequence of the second kind defined by $U_0=2, U_1=a$, $U_n=aU_{n-1}-bU_{n-2}$ for all $n\ge 2$ has the following properties:

(a) $U_n\equiv 0 \pmod {U_1}$ for all $n\equiv 1 \pmod 2$.

(b) $U_n\equiv a^n \pmod b$ for all $n\ge 1$.

(c) For every prime number $p$ and every $n\ge 1$ we have $U_n\equiv U_{n+p^2-1} \pmod p$.
\end{lemma}

\begin{proof}
The first two statements can be easily proved via induction. For the third statement, note that it would suffice the show that
$U_1\equiv U_{p^2}\pmod p$ and $U_2\equiv U_{p^2+1}\pmod p$.

One can then prove the full statement using induction on the recurrence relation that defines $U_n$.

We consider the case $p=2$ first. Since $U_1=a$ and $U_4=a^4 - 4a^2b + 2b^2$ it follows that
$U_4-U_1\equiv a^4-a \equiv 0 \pmod 2$. Similarly, since $U_2=a^2-2b$ and $U_5= a^5 - 5a^3b + 5ab^2$ we have that
$U_5-U_2\equiv a^5-a^2-5ab(a^2-b)+2b\equiv 0 \pmod 2$.

Consider next $p$ prime, $p\ge 3$.
Using the Binet formula it follows that
\begin{equation*}
U_{p^2}-U_1=\alpha^{p^2}+\beta^{p^2}-a = \left(\frac{a+\sqrt{\Delta}}{2}\right)^{p^2}+\left(\frac{a-\sqrt{\Delta}}{2}\right)^{p^2}-a.
\end{equation*}
Multiplying both sides by $2^{p^2}$ we obtain
\begin{align*}
2^{p^2}(U_{p^2}-U_1)&= \left(a+\sqrt{\Delta}\right)^{p^2}+\left(a-\sqrt{\Delta}\right)^{p^2}-2^{p^2}a=\\
&=\sum_{k=0}^{p^2} \binom{p^2}{k}a^{p^2-k}\left(\left(\sqrt{\Delta}\right)^k+\left(-\sqrt{\Delta}\right)^k\right)-2^{p^2}a=\\
&=\sum_{j=0}^{(p^2-1)/2}\binom{p^2}{2j}a^{p^2-2j}\cdot 2\left(\sqrt{\Delta}\right)^{2j}- 2^{p^2}a =\\
&=\sum_{j=0}^{(p^2-1)/2}\binom{p^2}{2j}a^{p^2-2j}\cdot 2\left(a^2-4b\right)^{j} -2^{p^2}a.
\end{align*}
Since $\binom{p^2}{2j}\equiv 0 \pmod p$ for all $1\le j\le (p^2-1)/2$, after reducing modulo $p$ the above equality becomes
\begin{equation*}
2^{p^2}(U_{p^2}-U_1)\equiv 2a^{p^2}-2^{p^2}a \equiv 0 \pmod p,
\end{equation*}
since from Fermat's Little Theorem we have $a^{p^2}\equiv a^p\equiv a \pmod p$ and  $2^{p^2}\equiv 2^p\equiv 2 \pmod p$.
It follows that $U_{p^2}\equiv U_1 \pmod p$ as claimed.

We proceed in a similar fashion with the second relation.
Using the Binet formula again it follows that
\begin{equation*}
U_{p^2+1}-U_2=\alpha^{p^2+1}+\beta^{p^2+1}-(a^2-2b) = \left(\frac{a+\sqrt{\Delta}}{2}\right)^{p^2+1}+\left(\frac{a-\sqrt{\Delta}}{2}\right)^{p^2+1}-(a^2-2b).
\end{equation*}
Multiplying both sides by $2^{p^2+1}$ we obtain
\begin{align*}
2^{p^2+1}(U_{p^2+1}-U_2)&= \left(a+\sqrt{\Delta}\right)^{p^2+1}+\left(a-\sqrt{\Delta}\right)^{p^2+1}-2^{p^2+1}(a^2-2b)=\\
&=\sum_{k=0}^{p^2+1} \binom{p^2+1}{k}a^{p^2+1-k}\left(\left(\sqrt{\Delta}\right)^k+\left(-\sqrt{\Delta}\right)^k\right)-2^{p^2+1}(a^2-2b)=\\
&=\sum_{j=0}^{(p^2+1)/2}\binom{p^2+1}{2j}a^{p^2+1-2j}\cdot 2\left(\sqrt{\Delta}\right)^{2j}- 2^{p^2+1}(a^2-2b) =\\
&=\sum_{j=0}^{(p^2+1)/2}\binom{p^2+1}{2j}a^{p^2+1-2j}\cdot 2\left(a^2-4b\right)^{j} -2^{p^2+1}(a^2-2b).
\end{align*}
Since $\binom{p^2+1}{2j}\equiv 0 \pmod p$ for all $0\le j\le (p^2-1)/2$, after reducing modulo $p$ the above equality becomes
\begin{equation}\label{V2}
2^{p^2+1}(U_{p^2+1}-U_2)\equiv 2a^{p^2+1}+2(a^2-4b)^{(p^2+1)/2} -2^{p^2+1}(a^2-2b) \mod p.
\end{equation}
However, since $p$ is odd we have that $x^{(p^2+1)/2}\equiv x \mod p$ for every integer $x$. Indeed, this is obviously true if $x\equiv 0 \pmod p$.
Otherwise, from the Fermat's Little Theorem it follows that $x^{p-1}\equiv 1 \pmod p$ from which $x^{(p^2-1)/2}\equiv 1 \pmod p$ and finally, $x^{(p^2+1)/2}\equiv x \pmod p$, as claimed. In particular, $a^{p^2+1}\equiv a^2 \pmod p$, $(a^2-4b)^{(p^2+1)}/2 \equiv a^2-4b \pmod p$, and $2^{p^2+1}\equiv 4 \pmod p$. Using these into equality \eqref{V2} we finally obtain that
\begin{equation*}
2^{p^2+1}(U_{p^2+1}-U_2)\equiv 2a^2+2(a^2-4b) -4(a^2-2b)\equiv 0 \mod p.
\end{equation*}
Eventually, $U_{p^2+1}\equiv U_2 \pmod p$, as claimed.
\end{proof}

A consequence of Lemma \ref{lemmaLucas} is that the sequence $\{U_n\}_{n\ge 1}$ is  purely periodic modulo any given prime and the period is $p^2-1$.
In particular, we have the following equalities.

\begin{equation}\label{p2}
U_n\equiv
\begin{cases}
a \pmod 2 \quad \text{if}\quad n\equiv 1, 2 \pmod 3,\\
a+ab \pmod 2 \quad \text{if}\quad n\equiv 0 \pmod 3,
\end{cases}
\end{equation}
and

\begin{equation}\label{p3}
U_n\equiv
\begin{cases}
a \pmod 3 \quad \text{if}\quad n\equiv 1, 3 \pmod 8,\\
a^2+b \pmod 3 \quad \text{if}\quad n\equiv 2, 6 \pmod 8,\\
2a^2b+a^2+2b^2 \pmod 3 \quad \text{if}\quad n\equiv 4 \pmod 8,\\
2ab^2+ab+a \pmod 3 \quad \text{if}\quad n\equiv 5, 7 \pmod 8,\\
2a^2b^2+a^2+2b^2 \pmod 3 \quad \text{if}\quad n\equiv 0 \pmod 8.
\end{cases}
\end{equation}
Similar tables can be computed for larger primes, although the expressions will be slightly more complicated.

The results above allow us to identify many quadratic Pisot numbers, $\alpha$, so that $\lfloor \alpha^n \rfloor$ is composite except for all $n\ge 2$.
In particular, if $a$ itself is composite then  $\lfloor \alpha^n \rfloor$ will be composite for all $n\ge 1$.

We split our discussion depending on the sign of $b$.

\begin{theorem}\label{thmbnegative}
Let $a$ and $b$ be integers so that $1\le -b \le a$. Then, the quadratic polynomial $P(z)=z^2-az +b$ is a Pisot polynomial, whose smaller root
$\beta = (a-\sqrt{a^2-4b})/2$ lies in the interval $(-1,0)$ and whose larger root $\alpha = (a+\sqrt{a^2-4b})/2$ is strictly greater than $1$. Moreover, if any of the conditions below is satisfied then $\lfloor \alpha^n \rfloor$ is
composite for all $n\ge 2$.

(a) $\gcd(a^2-1,b)=d>1$.

(b) $\gcd(a^2+1,b)=d>1$, $a\equiv 0 \pmod 3$, and $b\equiv 1 \pmod 3$.

(c) $\gcd(a^4+a^2+1,b)=d>1$, $a\equiv 1 \pmod 2$, and $b\equiv 1 \pmod 2$.

(d) $\gcd(a^4-a^2+1,b)=d>1$, $a\equiv 3\pmod 6$, and $b\equiv 1 \pmod 6$.
\end{theorem}
\begin{proof}
Start by noting that $P(-1)=1+a+b>0$, $P(0)=b<0$, and $P(1)=1-a+b<0$. It follows that $\beta \in (-1, 0)$ and $\alpha>1$, as claimed.
Let $U_n=\alpha^n+\beta^n$ be the Lucas sequence of the second kind satisfying $U_0=2$, $U_1=a$, and $U_n=aU_{n-1}-bU_{n-2}$ for all $n\ge 2$.

Since $-1<\beta<0$ it follows that
\begin{equation*}
\lfloor \alpha^n \rfloor=
\begin{cases}
U_n \quad \text{if}\,\, n\,\, \text{is odd},\\
U_n-1 \quad \text{if}\,\, n\,\, \text{is even}.
\end{cases}
\end{equation*}
Note that each of the four conditions (a)-(d) above require $|b|=-b\ge 2$ and implicitly $a\ge 2$.
From Lemma \ref{lemmaLucas} (a) it follows that $U_n\equiv 0 \pmod {a}$ whenever $n$ is odd. In particular, $\lfloor \alpha^n \rfloor\equiv 0 \pmod a$ for every odd $n$ and therefore since $a\ge 2$ it follows that $\lfloor \alpha^n \rfloor$ will be composite for all odd $n\ge 3$. The only possible exception is $n=1$ since in this case $\lfloor \alpha \rfloor=a$ and of course $a$ may very well be a prime itself.

It remains to check that $\lfloor \alpha^n \rfloor$ is composite for all $n$ even. Recall that in this case $\lfloor \alpha^n \rfloor=U_n-1$.

From Lemma \ref{lemmaLucas} (b) we have that $U_n\equiv a^n \pmod {b}$ for all $n\ge 1$. If $\gcd(a^2-1,b)=d>1$ then $a^n\equiv 1 \pmod d$ for all even $n$ and therefore $U_n\equiv 1 \pmod d$ for all $n$ even. But this implies that $\lfloor \alpha^n \rfloor\equiv 0 \pmod d$ for all such $n$, and therefore always composite. This proves part (a).

A specific example that fits into this class is $a=4, b=-3$. Then $gcd(a^2-1,b)=3>1$ and $\alpha=2+\sqrt{7}\approx 4.645\ldots$. Since $\lfloor \alpha \rfloor=4$ it follows that $\lfloor(2+\sqrt{7})^n\rfloor$ is composite for all $n\ge 1$.

If $\gcd(a^2+1,b)=d>1$ then $a^4\equiv 1 \pmod d$ and therefore by using Lemma \ref{lemmaLucas} (b) again
$U_{n}\equiv 1 \pmod d$ whenever $n\equiv 0 \pmod 4$.
On the other hand, equality \eqref{p3} tells us that $U_n\equiv a^2+b \equiv 1 \pmod 3$ for all $n\equiv 2 \pmod 4$.
It follows that in this case
\begin{equation*}
\lfloor \alpha^n \rfloor \equiv
\begin{cases}
0 \pmod a \quad \text{if}\,\, n \equiv 1 \pmod 2,\\
0 \pmod d \quad \text{if}\,\, n \equiv 0 \pmod 4,\\
0 \pmod 3 \quad \text{if}\,\, n \equiv 2 \pmod 4,
\end{cases}
\end{equation*}
which proves that $\lfloor \alpha^n \rfloor$ is composite for all $n\ge 2$. This proves part (b) of Theorem \ref{thmbnegative}.

A particular pair that fits within this class is $a=12, b=-5$. We have $\gcd(a^2+1,b)=5$ and $\alpha=6+\sqrt{41}\approx 12.403\ldots$.
Since $\lfloor \alpha \rfloor=12$ it follows that $\lfloor(6+\sqrt{41})^n\rfloor$ is composite for all $n\ge 1$.

The other two statements can be proved in very similar manner.
For part (c), since $gcd(a^4+a^2+1,b)=d>1$ it follows that $a^6\equiv 1 \pmod d$ from which $U_{n}\equiv 1 \pmod d$ whenever $n\equiv 0 \pmod 6$.
On the other hand, equality \eqref{p2} implies that  $U_n\equiv a \equiv 1 \pmod 2$ for all $n\equiv 1,2 \pmod 3$. So in this case,
\begin{equation*}
\lfloor \alpha^n \rfloor \equiv
\begin{cases}
0 \pmod a \quad \text{if}\,\, n \equiv 1 \pmod 2,\\
0 \pmod d \quad \text{if}\,\, n \equiv 0 \pmod 6,\\
0 \pmod 2 \quad \text{if}\,\, n \equiv 1,2 \pmod 3,
\end{cases}
\end{equation*}
which shows that $\lfloor \alpha^n \rfloor$ is composite for all $n\ge 2$.

An illustration for this case is obtained if taking $a=9, b=-7$. Then $gcd(a^4+a^2+1,b)=7$ and $\alpha=(9 + \sqrt{109})/2\approx 9.720\ldots$.
Again, since $\lfloor \alpha \rfloor=9$ it follows that $\lfloor((9 + \sqrt{107})/2)^n\rfloor$ is composite for all $n\ge 1$.

Finally, for part (d) note that the condition  $gcd(a^4-a^2+1,b)=d>1$ implies that $a^6\equiv -1 \pmod d$ from which $U_{n}\equiv 1 \pmod d$ whenever $n\equiv 0 \pmod {12}$. Since $a\equiv b \equiv 1\pmod 2$, equation \eqref{p2} implies that $U_n\equiv 1 \pmod {2}$ when $n\equiv 1, 2 \pmod 3$, while the fact that $a\equiv 0\pmod 3$ and $b\equiv 1 \pmod 3$ in conjunction with \eqref{p3} imply that  $U_n\equiv 1 \pmod {3}$ when $n\equiv 2 \pmod 4$.
So in this case we have
\begin{equation*}
\lfloor \alpha^n \rfloor \equiv
\begin{cases}
0 \pmod a \quad \text{if}\,\, n \equiv 1 \pmod 2,\\
0 \pmod d \quad \text{if}\,\, n \equiv 0 \pmod {12},\\
0 \pmod 2 \quad \text{if}\,\, n \equiv 1, 2 \pmod 3,\\
0 \pmod 3 \quad \text{if}\,\, n \equiv 2 \pmod 4,
\end{cases}
\end{equation*}
and since every positive integer $n$ falls within at least one of the congruence classes shown above it follows that $ \lfloor \alpha^n \rfloor$ is composite for all $n\ge 2$.

As an illustration of this case we can choose $a=171, b=-143$. Then, $\gcd(a^4-a^2+1,b)=13$ and $\alpha=(171+\sqrt{29813})/2\approx 171.832\ldots$. As in the previous cases, since $\lfloor \alpha \rfloor=171$ it follows that $\lfloor((171 + \sqrt{29813})/2)^n\rfloor$ is composite for all $n\ge 1$.

\end{proof}

 Next we address the situation when $b>0$.
\begin{theorem}\label{thmbpositive}
Let $a$ and $b$ be integers so that $1\le b \le a-2$. Then, the quadratic polynomial $P(z)=z^2-az +b$ is a Pisot polynomial, whose smaller root
$\beta = (a-\sqrt{a^2-4b})/2$ lies in the interval $(0,1)$ and whose larger root $\alpha = (a+\sqrt{a^2-4b})/2$ is strictly greater than $1$.  Moreover, if any of the conditions below is satisfied then $\lfloor \alpha^n \rfloor$ is composite for all $n\ge 2$.

(a) $\gcd(a-1,b)=d>1$.

(b) $\gcd(a^2+a+1,b)=d>1$, $a\equiv 1 \pmod 2$, and $b\equiv 1 \pmod 2$.

(c) $\gcd(a^2-a+1,b)=d>1$, $a\equiv 1 \pmod 6$, and $b\equiv 1 \pmod 6$.

(d) $\gcd(a+1,b)=d_1>1$, $\gcd(a^2+a+1,b)=d_2>1$, and $\gcd(a-1,b-1)=d_3>1$.
\end{theorem}

\begin{proof}

First, note that since $P(0)=b>0$ and $P(2)=4-2a+b<0$ it follows that $0<\beta<1$ and $\alpha>2$.
Also, since $\beta>0$ it follows that $\lfloor \alpha^n\rfloor = U_n-1$ for all $n\ge 1$.
The proof follows the same lines as that of Theorem \ref{thmbnegative}.

In case (a), since $a\equiv 1 \pmod d$ and $b\equiv 0 \pmod d$ it follows from Lemma \ref{lemmaLucas} (b) that $U_n\equiv a^n\equiv 1 \pmod d$ for all $n\ge 1$ and therefore $\lfloor \alpha^n \rfloor\equiv 0 \pmod d$ which implies that  $\lfloor \alpha^n \rfloor$ is composite for all $n\ge 2$.
One such example is given by $a=5, b=2$. Then $gcd(a-1,b)=2$ and $\alpha=(5+\sqrt{17})/2\approx 4.561\ldots$. Since, $\lfloor \alpha \rfloor=4$ it follows that $\lfloor \alpha^n \rfloor\equiv 0 \pmod 2$ and therefore composite for all $n\ge 1$.

In case (b), since $a\equiv 1\pmod 2$ it follows from \eqref{p2} that $U_n\equiv 1 \pmod 2$ for all $n\equiv 1,2 \pmod 3$. On the other hand, since $a^2-a+1\equiv 0 \pmod d$ we have that $a^3\equiv 1\pmod d$ which after using Lemma \ref{lemmaLucas} (b) gives that $U_n=\equiv 1 \pmod d$ for all $n\equiv 0 \pmod 3$. In summary,
\begin{equation*}
\lfloor \alpha^n \rfloor \equiv
\begin{cases}
0 \pmod 2 \quad \text{if}\,\, n \equiv 1,2 \pmod 3,\\
0 \pmod d \quad \text{if}\,\, n \equiv 0 \pmod 3,
\end{cases}
\end{equation*}
from which the conclusion follows.
One particular instance of this scenario is obtained when $a=9, b=7$. We have $\gcd(a^2+a+1,b)=7$ and $\alpha=(9+\sqrt{53})/2\approx 8.140\ldots$. Since $\lfloor \alpha \rfloor =8$ it follows that $\lfloor \alpha^n \rfloor$ is composite for all $n\ge 1$.
Another example is given by $a=9, b=7$. In this case, $\gcd(a^2+a+1,b)=7$ and $\alpha=(9+\sqrt{53})/2\approx 8.140\ldots$. Again, $\lfloor \alpha^n \rfloor$ is composite for all $n\ge 1$.

In case (c), the choices of $a$ and $b$ modulo $6$ in conjunction with \eqref{p2} and \eqref{p3} imply that $U_n\equiv 1 \pmod 2$ for all $n\equiv 1,2 \pmod 3$ and $U_n\equiv 1 \pmod 2$ for all $n\equiv 1,3,5,7 \pmod 8$. Also, since $a^2-a+1\equiv 0 \pmod d$ we have that $a^6\equiv 1\pmod d$ which after using Lemma \ref{lemmaLucas} (b) gives $U_n\equiv 1 \pmod d$ for all $n\equiv 0 \pmod 6$. Hence, in this case
\begin{equation*}
\lfloor \alpha^n \rfloor \equiv
\begin{cases}
0 \pmod 2 \quad \text{if}\,\, n \equiv 1,2 \pmod 3,\\
0 \pmod 3 \quad \text{if}\,\, n \equiv 1 \pmod 2,\\
0 \pmod d \quad \text{if}\,\, n \equiv 1 \pmod 6.
\end{cases}
\end{equation*}
One particular numerical example is obtained when choosing $a=19, b=7$.

We have $\gcd(a^2-a+1,b)=7$ and $\alpha=(19+3\sqrt{37})/2\approx 18.624\ldots$. Since $\lfloor \alpha \rfloor =18$ it follows that $\lfloor \alpha^n \rfloor$ is composite for all $n\ge 1$.

For case (d) we have $a^2\equiv 1 \pmod {d_1}$ and $b\equiv 0 \pmod {d_1}$ which after using Lemma \ref{lemmaLucas} (b) implies that $U_n\equiv 1 \pmod {d_1}$ for all $n\equiv 0 \pmod 2$.
Similarly, we have $a^3\equiv 1 \pmod {d_2}$ and $b\equiv 0 \pmod {d_2}$ which after using Lemma \ref{lemmaLucas} (b) again gives that $U_n\equiv 1 \pmod {d_2}$ for all $n\equiv 0 \pmod 3$.
Finally, since $a\equiv b\equiv 1\pmod {d_3}$ it is easy to check that the behavior of the sequence $\{U_n\}_{n\ge 1}$ modulo $d_3$ is as shown below
\begin{equation*}
U_n \pmod {d_3}: 1, -1, -2, -1, 1, 2, 1, -1, -2,  -1, 1, 2, 1, -1 \ldots
\end{equation*}
In particular, $U_n\equiv 1 \pmod {d_3}$ when $n\equiv 1, 5 \pmod 6$. Putting all of these together
\begin{equation*}
\lfloor \alpha^n \rfloor \equiv
\begin{cases}
0 \pmod {d_1} \quad \text{if}\,\, n \equiv 0 \pmod 2,\\
0 \pmod {d_2} \quad \text{if}\,\, n \equiv 0 \pmod 3,\\
0 \pmod {d_3} \quad \text{if}\,\, n \equiv 1,5 \pmod 6.
\end{cases}
\end{equation*}
As before, for every $n$, $\lfloor \alpha ^n \rfloor$ is divisible by at least one of the values $d_1$, $d_2$, or $d_3$.
As a specific example consider $a=86, b=21$. Then $d_1=3, d_2=7$, $d_3=5$ and $\alpha=43+2\sqrt{457}\approx 85.755\ldots$.
So, $\lfloor (43+ 2\sqrt{457})^n\rfloor$ is always composite. Note that this example does not fall into any of the previous three cases.

\end{proof}

In both Theorem \ref{thmbnegative} and Theorem \ref{thmbpositive} a necessary condition for ensuring $\lfloor \alpha^n\rfloor$ is always composite is that $|b|\ge 2$. It is reasonable to inquire whether such examples are possible if $b=1$ or $b=-1$. The answer is affirmative and it follows from a result of Cass \cite{cass}. We include the simple proof for the sake of completeness.

\begin{theorem}(\cite{cass})\label{cass}

(a) Let $\alpha$ be the larger root of the polynomial $z^2-az+1$ where $a\ge 3$ and let $h>1$ so that $\gcd(h,6)=1$. Then, $\lfloor (\alpha^h)^n\rfloor$ is composite for all $n\ge 2$.

(b) Let $\alpha$ be the larger root of the polynomial $z^2-az- 1$ where $a\ge 2$ and let $h>1$ so that $\gcd(h,6)=1$. Then, $\lfloor (\alpha^h)^n\rfloor$ is composite for all $n\ge 3$.
\end{theorem}

\begin{proof}
In case (a) we have $\alpha=(a+\sqrt{a^2-4})/2>2$ and $\beta=(a-\sqrt{a^2-4})/2 \in (0,1)$.
Let $U_n$ be the Lucas sequence of the second kind given by $U_0=2, U_1=a, U_n=aU_{n-1}-U_{n-2}$ for $n\ge 2$. Since $U_n=\alpha^n+\beta^n$ and $\beta \in (0,1)$ then $\lfloor \alpha^n\rfloor =U_n-1$ for all $n\ge 1$.

Computing the remainders of $U_n$ modulo $U_1-1=a-1$ for $n=1, 2, 3, \ldots$ we obtain that
\begin{equation*}
U_n \pmod{U_1-1} = 1, -1, -2, -1, 1, 2, 1, -1, -2, -1, 1, 2, \ldots
\end{equation*}
which implies that $U_n\equiv 1 \pmod {U_1-1}$ whenever $n\equiv 1, 5 \pmod 6$. In particular, it follows that $\lfloor \alpha^n\rfloor \equiv 0 \pmod{a-1}$ whenever $\gcd(n, 6)=1$.
Let $\{W_n\}_n$ be the Lucas sequence of the second kind given by $W_n=(\alpha^h)^n+(\beta^h)^n=U_{nh}$, $n\ge 0$. Applying the result above to this new sequence we obtain
\begin{equation*}
U_{nh}=W_n=(\alpha^n)^h+(\beta^n)^h\equiv 1 \pmod {\alpha^n+\beta^n-1}\equiv 1 \pmod {U_n-1},
\end{equation*}
which implies that $\lfloor \alpha^{nh}\rfloor \equiv 0 \pmod {U_n-1}$ for all $n\ge 1$ and therefore certainly composite for all $n\ge 2$.
The smallest example of this type is obtained if taking $a=3$ and $h=5$. Then $\alpha^h =((3+\sqrt{5})/2)^5=(123+55\sqrt{5})/2=122.991\ldots$.
It follows that $\lfloor((123+55\sqrt{5})/2)^n\rfloor$ is composite for all $n\ge 1$.

We next address part (b). In this case  $\alpha=(a+\sqrt{a^2+4})/2>2$ and $\beta=(a-\sqrt{a^2+4})/2 \in (-1,0)$. Let $U_n=\alpha^n+\beta^n$ as before.

If $n$ is odd then since $\beta<0$ we have $\lfloor\alpha^{nh}\rfloor=U_{nh}$ and by using Lemma \ref{lemmaLucas} (a) it follows that
$\lfloor\alpha^{nh}\rfloor\equiv 0 \pmod{a}$ and therefore composite for all odd $n\ge 3$.

On the other hand, $\alpha^2$ is the root of the quadratic $z^2-(a^2+2)z+1=0$. By applying the result from part (a) to $\alpha^2$ instead of $\alpha$ it follows that $\lfloor\alpha^{nh}\rfloor$ is composite for all even $n\ge 4$.

As a numerical illustration one can take $a=2$ and $h=5$. It follows that $\alpha^h=(1+\sqrt{2})^5=41+29\sqrt{5}=82.012\ldots$. Since $\lfloor \alpha^5 \rfloor=82$ and $\lfloor \alpha^{10} \rfloor=6725$ it follows that $\lfloor (41+29\sqrt{5})^n\rfloor$ is composite for all $n\ge 1$.
\end{proof}

\section{Cubic Pisot numbers and units in pure cubic fields}\label{sectioncubic}

\begin{theorem}\label{thmcubic}
Let $a, b$ integers such that $b\ge 0$ and $a\ge 2b+5$. Consider the polynomial
\begin{equation*}
P(z)=z^3-az^2+bz-1.
\end{equation*}
Then $P(z)$ is a Pisot polynomial whose largest root is no smaller than $5$. Moreover, if $a\equiv 3 \pmod 6$ and $b\equiv 3 \pmod 6$ then $\lfloor \alpha^n\rfloor$ is composite for all $n\ge 1$.
\end{theorem}

\begin{proof}
Start by computing $P(5)=125-25a+5b-1\le 125-25(2b+5)+5b-1=-45b-1<0$. So, $P$ must have a root $\alpha>5$.
Next, note that
\begin{equation*}
|az^2|>|z^3+bz-1| \quad \text{on the circle}\quad |z|=1/2.
\end{equation*}
Indeed, it would suffice to show that
\begin{equation*}
\frac{a}{4}>\frac{1}{8}+\frac{b}{2}+1\,\, \text{which is equivalent to}\,\, a>2b+4.5, \text{true by our choice of }\, a.
\end{equation*}
Using Rouch\'{e}'s theorem, it follows that the other two roots of $P$, $\beta$ and $\gamma$ satisfy $|\beta|<1/2$ and $|\gamma|<1/2$.
Hence, $P$ is a Pisot polynomial, as claimed.
Define, $U_n=\alpha^n+\beta^n+\gamma^n$ for all $n\ge 0$. Then $U_0=3, U_1=\alpha+\beta+\gamma=a, U_2=\alpha^2+\beta^2+\gamma^2=a^2-2b$, and
$U_n=aU_{n-1}-bU_{n-2}+U_{n-3}$ for all $n\ge 3$. If if $a\equiv 3 \pmod 6$ and $b\equiv 3 \pmod 6$ then via a quick induction one can easily prove that
$U_n\equiv 3 \pmod 6$ for all $n\ge 0$. Since $|\beta|<1/2$ and $|\gamma|<1/2$ it follows that $-1<\beta^n+\gamma^n<1$ for all $n\ge 1$ and therefore
$\lfloor \alpha^n \rfloor= U_n$  or $\lfloor \alpha^n \rfloor=U_n-1$ depending on whether $\beta^n+\gamma^n<0$ or $\beta^n+\gamma^n>0$.

It turns out that $\lfloor \alpha^n \rfloor\equiv 2, 3 \pmod 6$ and since $\alpha>5$ we conclude that  $\lfloor \alpha^n \rfloor$ is composite for all $n\ge 1$.
\end{proof}

Theorem \ref{thmcubic} has several immediate consequences as it allows us to test the fundamental unit of the pure cubic field $\mathbb{Q}(\sqrt[3]{m})$ for several values of $m$. From Dirichlet's Unit Theorem it follows that in any pure cubic field $\mathbb{Q}(\sqrt[3]{m})$, there exists a special unit, $\mu$, known as the fundamental unit, such that all other units $\rho$ are given by $\rho=\pm \mu^t$.
A table of fundamental units of $\mathbb{Q}(\sqrt[3]{m})$ for $2\le m\le 250$ is given in \cite{wada}.

For example, if $m=5$ then $\mu=41+24\sqrt[3]{5}+14\sqrt[3]{25}=122.975\ldots$. The minimal polynomial of $\mu$ is $z^3 - 123z^2 + 3z - 1$ whose coefficients satisfy the conditions of Theorem \ref{thmcubic}. Hence, $\gcd(\lfloor\mu^n\rfloor,6)>1$ for all $n\ge 1$.

A similar scenario takes place if $m=6$. In this case, $\mu= 109+60\sqrt[3]{6}+33\sqrt[3]{36}=326.990\ldots$ whose minimal polynomial is $z^3-327z^2+3z-1$ which again satisfies the requirements of Theorem \ref{thmcubic}. Hence, $\lfloor(109+60\sqrt[3]{6}+33\sqrt[3]{36})^n\rfloor$ is always composite as it is divisible by either $2$ or $3$ for every $n\ge 1$.

If $m=7$ then $\mu=4+2\sqrt[3]{7}+\sqrt[3]{49}=11.485\ldots$ and its minimal polynomial is $z^3-12z^2+6z-1$, which does not satisfy the hypotheses of Theorem \ref{thmcubic}. However, the minimal polynomial of $\mu^3$ is $z^3-1515z^2+3z-1$ and since this satisfies the requirements in Theorem \ref{thmcubic} we conclude that $\lfloor (4+2\sqrt[3]{7}+\sqrt[3]{49})^{3n}\rfloor$ is composite for all $n\ge 1$.

A similar situation happens when $m=3$. In this case $\mu=4+3\sqrt[3]{3}+2\sqrt[3]{9}=12.486\ldots$ satisfies the equation $z^3-12z^2-6z-1$, while $\mu^3$ has the minimal polynomial $z^3-1947z^2+3z-1$ which is of the type listed in Theorem \ref{thmcubic}. Hence, $\lfloor (4+3\sqrt[3]{3}+2\sqrt[3]{9})^{3n}\rfloor$ is composite for all $n\ge 1$.

One of the most interesting cases is $m=2$. We have $\mu=1+\sqrt[3]{2}+\sqrt[3]{4}=3.847\ldots$ whose minimal polynomial is $z^3-3x^2-3x-1$. One cannot directly apply Theorem \ref{thmcubic} and clearly $\lfloor \mu\rfloor=3$ which is a prime. On the other hand, the minimal polynomial of $\mu^2$ is $z^3-15z^2+3z-1$ for which Theorem \ref{thmcubic} can be applied. It follows that $\lfloor \mu^{2n}\rfloor$ is always divisible by either $2$ or $3$. This is exactly Huxley's result mentioned earlier. However, we can say more.

Let $\rho$ and $\overline{\rho}$ be the algebraic conjugates of $\mu$. We have $|\rho|=|\overline{\rho}|=\sqrt{2^{1/3}-1}=0.509\ldots$. Then clearly, $\rho^n+\overline{\rho}^n\in (-1,1)$ for all $n\ge 2$. Define $U_n=\mu^n+\rho^n+\overline{\rho}^n$. Since, $U_0=3, U_1=3, U_2=15$ and $U_n=3U_{n-1}+3U_{n-2}+U_{n-3}$ one can easily prove that $U_n\equiv 3\pmod 6$ and $U_n\ge 15$ for all $n\ge 2$. Since  $\lfloor\mu^n\rfloor=U_n$ or $\lfloor\mu^n\rfloor=U_n-1$ for all $n\ge 2$ it follows that $\lfloor(1+\sqrt[3]{2}+\sqrt[3]{4})^n \rfloor$ is composite for all $n\ge 2$.

It would be interesting to decide for which values of $m$ is it true that the sequence $\{\lfloor \mu^n\rfloor\}_{n\ge 1}$ is prime-free (or that is has at most finitely many prime terms). We present one such result below.

\begin{theorem}\label{cubiceven}
Let $m$ be a positive cube-free even integer, with $m\not\equiv \pm 1 \pmod 9$ and let $\mu>1$ be the fundamental unit in $\mathbb{Q}(\sqrt[3]{m})$. Then $\lfloor \mu^n \rfloor$ is composite for all $n\ge 4$.
\end{theorem}

\begin{proof}
Start by writing $m$ as $m=hk^2$ where $h$ and $k$ are coprime square-free positive integers. Obviously, every cube-free $m$ has a unique such factorization.

An old theorem of Dedekind (see e. g. \cite{aw} Thoerem 7.3.2) states that if $m\not\equiv \pm 1\pmod 9$, then an integral basis in $\mathbb{Q}(\sqrt[3]{m})$ is $\{1, \sqrt[3]{hk^2},  \sqrt[3]{h^2k}\}$.

Hence, there exist integers $c_1,c_2,c_3$ such that
\begin{equation*}
\mu= c_1+c_2\sqrt[3]{hk^2}+c_3\sqrt[3]{h^2k}.
\end{equation*}
It is known that the algebraic conjugates of $\mu$ are
\begin{equation*}
\rho= c_1+c_2\omega\sqrt[3]{hk^2}+c_3\omega^2\sqrt[3]{h^2k} \,\, \text{and}\,\,\overline\rho= c_1+c_2\omega^2\sqrt[3]{hk^2}+c_3\omega\sqrt[3]{h^2k},
\end{equation*}
where
$\omega=(-1+i\sqrt{3})/2$.
Computing the elementary symmetric polynomials in $\mu$, $\rho$, $\overline\rho$ we obtain
\begin{align*}
a&=\mu+\rho+\overline\rho=3c_1,\\
b&=\mu\rho+\mu\overline\rho+\rho\overline\rho= 3c_1^2-3c_2c_3hk,\\
c&=\mu\rho\overline\rho=c_1^3+c_2^3hk^2+c_3^3h^2k-3c_1c_2c_3hk=1.
\end{align*}

Since $c=1$ and $hk\equiv 0 \pmod 2$ it follows that $c_1\equiv 1 \pmod 2$. Note that this immediately implies that $a\equiv 3 \pmod 6$ and $b\equiv 3 \pmod 6$. Define $U_n=\mu^n+\rho^n+\overline\rho^n$; as before $U_0=3$, $U_1=a$, $U_2=a^2-2b$, and $U_n=aU_{n-1}-bU_{n-2}+U_{n-3}$ which leads to $U_n\equiv 3\pmod 6$ for all $n\ge 0$.

Salem \cite{salem} conjectured and shortly after Siegel \cite{siegel} proved that the minimal Pisot number is the real root of $z^3-z-1=0$. This number is usually denoted by $\theta_1$ and it is known as the plastic number.
\begin{equation*}
\theta_1=\sqrt[3]{\frac{9+\sqrt{69}}{18}}+\sqrt[3]{\frac{9-\sqrt{69}}{18}}=1.3247\ldots.
\end{equation*}

Also, Siegel \cite{siegel} proved that the second smallest Pisot number is the positive root of $z^4-z^3-1=0$, $\theta_2=1.380\ldots$, and that all other Pisot numbers are $>\sqrt{2}$.
Since in our case it is obvious that $\mu\neq \theta_1, \theta_2$ we obtain that
\begin{equation*}
|\rho|^2=\rho\overline\rho=\frac{1}{\mu}\le \frac{1}{\sqrt{2}} \,\, \text{from which} \,\, |\rho|^4 <\frac{1}{2}.
\end{equation*}
So, for every $n\ge 4$ we have $-1<\rho^n+\overline\rho^n<1$ from which $\lfloor\mu^n\rfloor=U_n$ or $\lfloor\mu^n\rfloor=U_n-1$ for all such $n$.
Moreover, we have $\mu^4> (\sqrt{2})^4= 4.$
Since $U_n\equiv 3 \pmod 6$ and $\lfloor \mu^n \rfloor\ge 4$ it follows that $\lfloor \mu^n \rfloor$ is composite for all $n\ge 4$, as claimed.
\end{proof}

Note that the conclusion of Theorem \ref{cubiceven} may be true even when $m$ is odd or $m\equiv \pm 1 \mod 9$.
We have already seen that in $\mathbb{Q}(\sqrt[3]{5})$ the fundamental unit $\mu=41+24\sqrt[3]{5}+14\sqrt[3]{25}$ has minimal polynomial $z^3-123z^2+3z-1$  so Theorem \ref{thmcubic} applies.
Similar situations hold in $\mathbb{Q}(\sqrt[3]{11})$ and  $\mathbb{Q}(\sqrt[3]{15})$ where  the fundamental units are $\mu= 89+40\sqrt[3]{11}+18\sqrt[3]{121}$ and $\mu=5401+2190\sqrt[3]{15}+888\sqrt[3]{225}$ with the minimal polynomials being $z^3-267z^2+3z-1$ and $z^3-16203z^2+3z-1$, respectively. Again, one can use Theorem \ref{thmcubic} in each of these cases.

The situation seems to be more complicated when $m\equiv \pm 1 \pmod 9$ as in general the expression of the integral basis of $\mathbb{Q}(\sqrt[3]{m})$ is rather convoluted -see e. g. Dedekind's result, Theorem 7.3.2 in \cite{aw}.
However, if $m=26$ things look quite tame. The fundamental unit is $\mu= 9+3\sqrt[3]{26}+\sqrt[3]{676}$ and its minimal polynomial is $z^3-27x^2+9z-1$. One can certainly use Theorem \ref{thmcubic} is this case.

Based on our admittedly limited verifications, we venture the following.
\begin{conjecture}
Let $m$ be a cubefree positive integer $m>1$ and let $\mu>1$ be the fundamental unit in $\mathbb{Q}(\sqrt[3]{m})$. Then, for every $n$ large enough $gcd(\lfloor \mu^{6n}\rfloor,6)>1$.
\end{conjecture}

\section{Biquadratic Pisot numbers}\label{sectionbiquadratic}

Let $p$ and $q$ be coprime squarefree integers greater than $1$. The field formed by adjoining $\sqrt{p}$ and $\sqrt{q}$ to $\mathbb{Q}$ is denoted $\mathbb{Q}(\sqrt{p},\sqrt{q})$. Since $\mathbb{Q}(\sqrt{p},\sqrt{q}) = \mathbb{Q}(\sqrt{p}+\sqrt{q})$  and the minimal polynomial of $\sqrt{p}+\sqrt{q}$
is $z^4- 2(p+q)z^2+(p-q)^2$, $\mathbb{Q}(\sqrt{p},\sqrt{q})$ is a biquadratic field over $\mathbb{Q}$.

The elements of $\mathbb{Q}(\sqrt{p},\sqrt{q})$ are numbers of the form $r_1+r_2\sqrt{p}+r_3\sqrt{q}+r_4\sqrt{pq}$, where $r_1,r_2,r_3,r_4 \in \mathbb{Q}$.
Any element of $\mathbb{Q}(\sqrt{p},\sqrt{q})$ which satisfies a monic equation with rational integer coefficients os called an integer in $\mathbb{Q}(\sqrt{p},\sqrt{q})$.

Williams \cite{williams} proved that the integers in $\mathbb{Q}(\sqrt{p},\sqrt{q})$ are necessarily of the form
\begin{equation*}
\frac{1}{4}\left(c_1+c_2\sqrt{p}+c_3\sqrt{q}+c_4\sqrt{p}\sqrt{q}\right)
\end{equation*}
where $c_1,c_2,c_3,c_4$ are rational integers which satisfy additional equalities modulo $4$ depending on the values of $p$ and $q$ modulo $4$.

While we are not aware of a simple criterion for generating biquadratic Pisot numbers, it is rather straightforward to perform a computer search due to the following result.

\begin{lemma}\label{lemmabiquadratic}
Let $p$ and $q$ be coprime squarefree integers greater than $1$ and let
\begin{equation*}
\alpha=\frac{1}{4}\left(c_1+c_2\sqrt{p}+c_3\sqrt{q}+c_4\sqrt{p}\sqrt{q}\right)
\end{equation*}
be a Pisot integer in $\mathbb{Q}(\sqrt{p},\sqrt{q})$.
Then
\begin{equation*}
|c_1- c_4\sqrt{p}\sqrt{q}|<4,\, |c_2< c_4\sqrt{q}|<\frac{4}{\sqrt{p}},\,\,\, \text{and}\,\,\,|c_3- c_4\sqrt{p}|<\frac{4}{\sqrt{q}}.
\end{equation*}
\end{lemma}

\begin{proof}
The conjugates of $\alpha$ over $\mathbb{Q}$ can be written as
\begin{align*}
\beta&=\frac{1}{4}\left(c_1+c_2\sqrt{p}-c_3\sqrt{q}-c_4\sqrt{p}\sqrt{q}\right),\\
\gamma&=\frac{1}{4}\left(c_1-c_2\sqrt{p}+c_3\sqrt{q}-c_4\sqrt{p}\sqrt{q}\right),\\
\delta&=\frac{1}{4}\left(c_1-c_2\sqrt{p}-c_3\sqrt{q}+c_4\sqrt{p}\sqrt{q}\right).
\end{align*}

Since $\alpha$ is a Pisot number it follows that $\beta, \gamma, \delta \in (-1,1)$.
The three inequalities readily follow from the following equalities
\begin{equation*}
2(\beta+\gamma)= c_1-c_4\sqrt{p}\sqrt{q}, \frac{2(\beta-\delta)}{\sqrt{p}}=c_2-c_4\sqrt{q}, \,\, \text{and}\,\,\frac{2(\gamma-\delta)}{\sqrt{q}}=c_3-c_4\sqrt{p}.
\end{equation*}
\end{proof}

\begin{theorem}\label{theorembiquadratic}
Let $\alpha$ be a (biquadratic) Pisot number whose minimal polynomial is
\begin{equation*}
P(z)=z^4-az^3+bz^2-cz+d.
\end{equation*}
Suppose that $a\equiv 4 \pmod 6$, $b\equiv 0 \pmod 3$, $c\equiv 4\pmod 6$ and $d\equiv 1 \pmod 3$.
Then for all large enough $n$, $\lfloor\alpha^n\rfloor \equiv 3, 4 \pmod 6$, and therefore composite.
\end{theorem}
\begin{proof}
Let $\beta, \gamma, \delta$ be the algebraic conjugates of $\alpha$. For all $n\ge 0$ define $U_n=\alpha^n+\beta^n+\gamma^n+\delta^n$.

Then $U_0=4, U_1=a, U_2= a^2-2ab, U_3=a^3-3ab+3c$ and for all $n\ge 4$ we have $U_n=aU_{n-1}-bU_{n-2}+cU_{n-3}-dU_{n-4}$.
The modular restrictions on the coefficients of the minimal polynomial $P$ imply that $U_n\equiv 4\pmod 6$ for $0\le n\le 3$ and
an easy induction proves that $U_n\equiv 4 \pmod 6$ for all $n\ge 0$.
Since $\alpha$ is Pisot it follows that $\beta^n+\gamma^n+\delta^n$ belongs to the interval $(-1,1)$ for all sufficiently large $n$.
Eventually, this implies that  $\lfloor\alpha^n\rfloor \equiv 3, 4 \pmod 6$ for all large enough $n$, as claimed.
\end{proof}
Lemma \ref{lemmabiquadratic} allows us to identify many biquadratic Pisot numbers which satisfy the conditions from Theorem \ref{theorembiquadratic}.
In Table \ref{table1} we list all such numbers smaller than $100$ that we could find. For all these entries $\lfloor\alpha^n\rfloor \equiv 3, 4 \pmod 6$ as soon as $n\ge 3$.

\begin{table}[hbt!]
\begin{centering}
\begin{tabular}{c|c|c}
 $\alpha$ (approximate)           & $\alpha$             & Minimal Polynomial      \\
\midrule
\midrule
$16.881\ldots$ & $(8 + 5\sqrt{3} + 3\sqrt{7} + 2\sqrt{3}\sqrt{7})/2$ & $z^4 - 16z^3 - 15z^2 + 2z + 1$\\
\midrule
$46.120\ldots$ & $(23 + 13\sqrt{3} + 9\sqrt{7} + 5\sqrt{3}\sqrt{7})/2$, & $z^4 - 46z^3 - 6z^2 + 20z + 4$\\
\midrule
$52.962\ldots$ & $(26 + 15\sqrt{3} + 12\sqrt{5} + 7\sqrt{3}\sqrt{5})/2$ & $z^4 - 52z^3 - 51z^2 + 2z + 1$\\
\midrule
$53.581\ldots$ & $(26 + 12\sqrt{5} + 11\sqrt{6} + 5\sqrt{5}\sqrt{6})/2$ & $z^4 - 52z^3 - 84z^2 - 40z - 5$\\
\midrule
$58.508\ldots$ & $(29 + 17\sqrt{3} + 9\sqrt{11}+ 5\sqrt{3}\sqrt{11})/2$ & $z^4 - 58z^3 - 30z^2 + 14z + 1$\\
\midrule
$63.998\ldots$ & $(32 + 13\sqrt{6} + 12\sqrt{7} + 5\sqrt{6}\sqrt{7})/2$ & $z^4 - 64z^3 + 8z + 1$\\
\midrule
$64.786\ldots$ & $(32+9\sqrt{13}+7\sqrt{21}+2\sqrt{13}\sqrt{21})/2$ & $z^4-64z^3-51z^2+2z+4$\\
\midrule
$75.515\ldots$ & $(38 + 27\sqrt{2} + 22\sqrt{3} + 15\sqrt{2}\sqrt{3})/2$ &  $z^4 - 76z^3 + 36z^2 + 44z - 23$\\
\midrule
$76.627\ldots$ & $(38 + 27\sqrt{2} + 10\sqrt{15} + 7\sqrt{2}\sqrt{15})/2$ &  $z^4 - 76z^3 - 48z^2 - 4z + 1$\\
\midrule
$81.740\ldots$ & $(41 + 18\sqrt{5} + 9\sqrt{21} + 4\sqrt{5}\sqrt{21})/2$ & $z^4 - 82z^3 + 21z^2 + 20z - 5$\\
\midrule
$87.141\ldots$ & $(44 + 25\sqrt{3} + 12\sqrt{13} + 7\sqrt{3}\sqrt{13})/2$ & $z^4 - 88z^3 + 75z^2 - 16z + 1$\\
\midrule
$94.030\ldots$ & $(47 + 33\sqrt{2} + 21\sqrt{5} + 15\sqrt{2}\sqrt{5})/2$ & $z^4 - 94z^3 - 3z^2 + 14z + 1$\\
\midrule
$99.848\ldots$ & $(50 + 29\sqrt{3} + 12\sqrt{17} + 7\sqrt{3}\sqrt{17})/2$ & $z^4 - 100z^3 + 15z^2 + 14z - 2$\\
\midrule
\midrule
\end{tabular}
\caption{\small{Biquadratic Pisot numbers $\alpha< 100$ with the property that $\lfloor \alpha^n\rfloor\equiv 3, 4$ for all $n\ge 3$.}}
\label{table1}
\end{centering}
\end{table}

We conjecture that there are infinitely many biquadratic Pisot numbers with the property described above.

\section{Pisot numbers of arbitrary degree}\label{sectiondegreed}

So far we exhibited Pisot numbers of degrees $2$, $3$, and $4$ which have the property that the integer part of any power is prime for at most finitely many $n$. In \cite{ad}, Alkauskas and Dubickas constructed Pisot polynomials of arbitrary degree $d$ whose largest root $\alpha$ has the property that $\lfloor \alpha^n\rfloor$ is even for all $n\ge 1$. We generalize their result in the following Theorem.

\begin{theorem}\label{thmpisotd}
Let $d\ge 2, p\ge 2$, and $m\ge d^{d-1}$ be integers. Consider the polynomial
\begin{equation*}
P(z)=z^d -(3mp+1)z^{d-1}+2mpz^{d-2}-p.
\end{equation*}
Then, $P$ is a Pisot polynomial and $\alpha_1$, the largest root of $P$, has the property that $\alpha_1>3mp$ and $\lfloor \alpha_1^n\rfloor \equiv 0 \pmod p$ for all $n\ge 1$. In particular, $\lfloor \alpha^n\rfloor$ is always composite.
\end{theorem}

\begin{proof}

Note that if $d= 2$, the polynomial becomes $P(z)=z^2-(3mp+1)z+(2m-1)p$. Since, $3mp+1\ge (2m-1)p+2$ and $gcd(3mp+1-1, (2m-1)p)=p\ge 2$ the conclusion is a consequence of Theorem \ref{thmbpositive} (a).
Hence, for the rest of the proof we will assume that $d\ge 3$.

The first part of the proof deals with the locations of the roots of $P$. More precisely, we will prove the following three assertions:

\begin{claim}\label{claim1} There exists a root $\alpha_1$ in the interval $(3mp, 3mp+1)$. \end{claim}

\begin{claim}\label{claim2} There exists a root $\alpha_2$ in the interval $(1/2, 1-1/(2d))$.\end{claim}

\begin{claim}\label{claim3} The remaining roots, $\alpha_3, \alpha_4,\ldots, \alpha_d$, are all inside the circle $|z|=1/d$.\end{claim}

One can readily check that
\begin{equation*}
P(3mp)=-mp(3mp)^{d-2}-p<0 \quad \text{and} \quad P(3mp+1)=2mp(3mp+1)^{d-2}-p>0
\end{equation*}
which implies there existence of a root $\alpha_1$ in the interval $(3mp, 3mp+1)$. This proves Claim \ref{claim1}. Note that $\lfloor \alpha_1 \rfloor=3mp \equiv 0 \pmod p$.

We address Claim \ref{claim2} next.

On one hand we have
\begin{equation*}
P\left(\frac{1}{2}\right)=p\left(\frac{m}{2^{d-1}}-1\right)-\frac{1}{2^d}>p\left(\left(\frac{d}{2}\right)^{d-1}-1\right)-\frac{1}{2^d}\ge p\left(\left(\frac{3}{2}\right)^2-1\right)-\frac{1}{2^3}>0.
\end{equation*}

On the other hand,
\begin{align*}
P\left(1-\frac{1}{2d}\right)=&\left(1-\frac{1}{2d}\right)^{-d+2}\cdot\left(-mp+\frac{3mp}{2d}-\frac{1}{2d}+\frac{1}{4d^2}\right)-p<\\
<&\left(1-\frac{1}{2d}\right)^{-d+2}\cdot\left(-mp+\frac{2mp}{2d}\right)=mp\left(1-\frac{1}{2d}\right)^{-d+2}\cdot\left(-1+\frac{3}{2d}\right)<0.
\end{align*}
This proves Claim \ref{claim2}.

Finally, notice that
\begin{equation}\label{rouche}
|2mpz^{d-2}|> |z^d-(3mp+1)z^{d-1}-p| \quad \text{on the circle}\quad |z|=\frac{1}{d}.
\end{equation}

Indeed, it would suffice to show that
\begin{equation*}
\frac{2mp}{d^{d-2}}>\frac{1}{d^d}+\frac{3mp+1}{d^{d-1}}+p
\end{equation*}
which is equivalent to proving
\begin{equation*}
(2d-3)\frac{m}{d^{d-1}}>\frac{1}{p\,d^{d}}+\frac{1}{p\,d^{d-1}}+1.
\end{equation*}
However, since $m\ge d^{d-1}$, the left hand term is no smaller than $2d-3\ge 3$, while the right hand term is certainly smaller than $3$.
Thus, inequality \eqref{rouche} is valid.
Then, by Rouch\'{e}'s theorem it follows that $z^{d-2}$ and $P(z)$ have the same number of roots inside the circle $|z|=1/d$. It follows that
the remaining $d-2$ roots of $P$ have modulus smaller than $1/d$ and therefore Claim \ref{claim3} is proved.

\begin{claim}\label{claim4}
For every $n\ge 2$ we have
\begin{equation*}
0<\alpha_2^n+\alpha_3^n+\cdots +\alpha_d^n <1.
\end{equation*}
\end{claim}

From Claim \ref{claim2} and Claim \ref{claim3} it follows that
\begin{equation*}
\frac{1}{2^n}<\alpha_2^n<\left(1-\frac{1}{2d}\right)^n,\quad \text{and} \quad -\frac{d-2}{d^n}<\alpha_3^n+\alpha_4^n+\cdots +\alpha_d^n <\frac{d-2}{d^n},
\end{equation*}
from which we obtain
\begin{equation*}
\frac{1}{2^n}-\frac{d-2}{d^n}<\alpha_2^n+\alpha_3^n+\cdots +\alpha_d^n<\frac{d-2}{d^n}+\left(1-\frac{1}{2d}\right)^n.
\end{equation*}
It would therefore suffice to show that for all $n\ge 2$
\begin{equation*}
\frac{1}{2^n}-\frac{d-2}{d^n}>0 \quad \text{and} \quad \frac{d-2}{d^n}+\left(1-\frac{1}{2d}\right)^n<1.
\end{equation*}
The first inequality is equivalent to $d/2>(d-2)^{1/n}$ which is obviously true since already $d/2>(d-2)^{1/2}$ for all $d\ge 3$.
For the second inequality, note that the sequence with general term $x_n=\frac{d-2}{d^n}+\left(1-\frac{1}{2d}\right)^n$ is strictly decreasing.
Since $x_2=1-\frac{7}{4d^2}<1$ the proof of Claim \ref{claim4} is complete.

\begin{claim}\label{claim5}
For every $n\ge 1$ let $U_n=\alpha_1^n+\alpha_2^n+\alpha_3^n+\cdots \alpha_d^n$. Then, $U_n\equiv 1 \pmod {p}$.
\end{claim}

For every $1\le j \le d$, let  $\sigma_j$ be the $j$-th elementary symmetric polynomial in the $d$ variables $\alpha_1, \alpha_2, \ldots, \alpha_d$.
Note that by the definition of $P$ we have
\begin{align*}
\sigma_1&=\alpha_1+\alpha_2+\cdots +\alpha_d =3mp+1.\\
\sigma_2&=\sum_{1\le i_1<i_2} \alpha_{i_1}\alpha_{i_2}=2mp,\\
\sigma_3&=\sigma_4=\ldots = \sigma_{d-1}=0, \text{and}\\
\sigma_d&=\alpha_1\alpha_2\cdots\alpha_d=(-1)^{d-1}p.
\end{align*}
Note that $\sigma_1\equiv 1\pmod p$ while $\sigma_j\equiv 0 \pmod p$ for all $2\le j\le d$.
Waring's formula allows us to express $U_n$ in terms of $\sigma_1, \sigma_2, \ldots, \sigma_d$
\begin{equation*}
U_n=\sum c(i_1,i_2,\ldots, i_d)\,\sigma_1^{i_1}\sigma_2^{i_2}\cdots\sigma_d^{i_d}
\end{equation*}
where
\begin{equation*}
c(i_1,i_2,\ldots, i_d)=(-1)^{(i_2+i_4+i_6+\cdots)}\frac{(i_1+i_2+\cdots+i_d-1)!}{{i_1}!{i_2}!\cdots{i_d}!}\,n
\end{equation*}
are all integer coefficients and the summation is taken over all $d$-tuples $(i_1,i_2,\ldots, i_d)$ with nonnegative integer entries such that $i_1+2i_2+3i_3+\cdots +di_d=n$.
Since $\sigma_j\equiv 0 \pmod p$ for all $2\le j\le d$ it follows that
\begin{equation*}
U_n\equiv c(n, 0,0,\ldots, 0) \sigma_1^n \pmod{p} \quad \text{that is} \quad U_n\equiv 1\pmod p, \quad \text{as asserted}.
\end{equation*}

We are now in position to complete the proof of Theorem \ref{thmpisotd}.

As noted earlier $\lfloor \alpha_1 \rfloor=3mp\equiv 0\pmod p$.
From Claim \ref{claim4} it follows that for all $n\ge 2$ we have $\lfloor \alpha_1^n \rfloor=U_n-1$. Using Claim \ref{claim5}, we conclude that
$\lfloor \alpha_1^n \rfloor\equiv 0 \pmod p$.

\end{proof}

\section{\bf Conclusions and open questions}\label{conclusions}

In the previous sections we described many instances of numbers $\alpha$ for which the sequence with general term $\lfloor \alpha^n \rfloor$ is composite for all but possibly finitely many integers $n\ge 1$. All our examples were carefully chosen Pisot numbers. It is therefore reasonable to inquire whether it is possible to exhibit an explicit non-Pisot number with the same property. Can one find a rational number or a transcendental number that satisfies this condition?

From the other end, does there exist a number $\alpha>1$ such that $\lfloor \alpha^n \rfloor$ is prime for all but finitely many $n$?

{\footnotesize{

}}

\end{document}